\documentclass[reqno,12pt]{amsart}
\usepackage{amsmath,amssymb,latexsym,soul,cite,mathrsfs}
\usepackage{color,enumitem,graphicx}
\usepackage[colorlinks=true,urlcolor=blue,
citecolor=red,linkcolor=blue,linktocpage,pdfpagelabels,
bookmarksnumbered,bookmarksopen]{hyperref}
\usepackage[english]{babel}
\usepackage[left=2.7cm,right=2.7cm,top=3.2cm,bottom=3.2cm]{geometry}

\usepackage[hyperpageref]{backref}
\numberwithin{equation}{section}

\newtheorem{theorem}{Theorem}[section]
\newtheorem{lemma}{Lemma}[section]
\newtheorem{corollary}{Corollary}[section]
\newtheorem{proposition}{Proposition}[section]

\newtheorem{remark}{Remark}[section]
\newtheorem{definition}{Definition}[section]

\title[Regularity of stable solutions to quasilinear elliptic equations]
 {Regularity of stable solutions to quasilinear elliptic equations on Riemannian models}

\author[J.M.\ do \'O]{Jo\~ao Marcos do \'O}
\author[R.G.\ Clemente]{Rodrigo G.\ Clemente}

\address[J.M. do \'O]{Department of Mathematics, Bras\'{\i}lia University
	\newline\indent
	70910-900, Bras\'{\i}lia, DF, Brazil}
\email{\href{mailto:jmbo@pq.cnpq.br}{jmbo@pq.cnpq.br}}

\address[R.\ Clemente]{Department of Mathematics, 
 Rural Federal University of Pernambuco
\newline\indent 
52171-900, Recife, Pernambuco, Brazil}
\email{\href{mailto:rodrigo.clemente@ufrpe.br}{rodrigo.clemente@ufrpe.br}}

\thanks{Research supported in part by INCTmat/MCT/Brazil, CNPq and CAPES}
\thanks{Corresponding author: J. M. do \'O}

\usepackage[norefs,nocites]{refcheck}

\begin{document}
	
%
%

	\begin{abstract}

	 We investigate the regularity of semi-stable, radially symmetric, and decreasing solutions for a class of quasilinear reaction-diffusion equations in the inhomogeneous context of Riemannian manifolds. 
	 We prove uniform boundedness, Lebesgue and Sobolev estimates for this class of solutions for equations involving the p-Laplace Beltrami operator and locally Lipschitz non-linearity.
	 We emphasize that our results do not depend on the boundary conditions and the specific form of the non-linearities and metric. 
	 Moreover, as an application, we establish regularity of the extremal solutions for equations involving the p-Laplace Beltrami operator with zero Dirichlet boundary conditions.
	\end{abstract}

\maketitle

%
%

\section{Introduction}
Let $(\mathcal{M},g)$ be a Riemannian model of dimension $N\geq 2$, that is, a manifold $\mathcal{M}$ admitting a pole $\mathcal{O}$ and whose metric $g$ is given, in polar coordinates around $\mathcal{O}$, by
\begin{equation}\label{09}
\mathrm{d} s^2=\mathrm{d} r^2+\psi(r)^2\mathrm{d} \theta^2 \quad \text{for }r \in (0,R)\text{ and }\theta\in\mathbb{S}^{N-1},
\end{equation}
where $r$ is by construction the Riemannian distance between the point $P=(r,\theta)$ to the pole $\mathcal{O}$, $\mathrm{d}\theta^2$ is the canonical metric on the unit sphere $\mathbb{S}^{N-1}$ and $\psi$ is a smooth function in $[0,R)$ and positive in $(0,R)$ for some $R\in(0,+\infty]$, and $\psi(0)=\psi^{\prime\prime}(0)=0\text{ and }\psi^\prime(0)=1$. As examples we have the important cases of space forms, i.e., the unique complete and simply connected Riemannian manifold of constant sectional curvature $K_\psi$ corresponding to the choice of $\psi$ namely,
\begin{flalign}\label{43}
\begin{array}{lccllll}
(i)   &   \psi(r) & = & \sinh r,  & \quad  K_\psi=-1 & \quad \left(\text{Hyperbolic space}\right)\\ 
(ii)  &   \psi(r) & = & r,        & \quad  K_\psi=0  & \quad \left(\text{Euclidean space}\right)\\
(iii) &   \psi(r) & = & \sin r,   & \quad  K_\psi=1  & \quad \left(\text{Elliptic space}\right)
\end{array}&&
\end{flalign}

Let us denote the geodesic ball of radius r with center at the pole $\mathcal{O}$ by $\mathcal{B}_r$ and $W_r^{1,p}(\mathcal{B}_1)$ the elements of the Sobolev space which are radially symmetric with respect to the pole $\mathcal{O}$. For $u\in W_{r}^{1,p}(\mathcal{B}_1)$, let us consider the energy functional
\begin{equation}\label{12}
J_\delta(u)=\frac{1}{p}\int_{\mathcal{B}_1\setminus\overline{\mathcal{B}_\delta}}|\nabla_g u|^p\,\mathrm{d} v_g-\int_{\mathcal{B}_1\setminus\overline{\mathcal{B}_\delta}}F(u)\,\mathrm{d} v_g, \text{ where }F(x,t)=\int_{0}^{t}f(x,s)\mathrm{d}s.
\end{equation}

\begin{definition}
	We say that a decreasing function $u\in W_{r}^{1,p}(\mathcal{B}_1)$ is a radial local minimizer of \eqref{12} if for any $0< \delta<1$ there exists $\epsilon=\epsilon(\delta)>0$ such that for all radial functions $\phi\in C_{c}^{1}\left(\mathcal{B}_1\setminus\overline{\mathcal{B}_\delta}\right)$ satisfying $\Vert \phi \Vert_{C^1}\leq \epsilon$, we have
	\[
	J_\delta(u)\leq J_\delta(u+\phi).
	\]
\end{definition}

\begin{definition}
	Let $u\in W_{r}^{1,p}(\mathcal{B}_1)$ be a critical point of \eqref{12}. We say that $u$ is semi-stable if $u_r(r) < 0$ for all $r\in (0,1)$ and for all radially symmetric function $\xi\in C_{c}^{1}(\mathcal{B}_1\setminus\left\{\mathcal{O}\right\})$ it holds
	\begin{equation}\label{04}
	\int_{\mathcal{B}_1}\left[(p-1)|\nabla u|^{p-2}|\nabla\xi|^2-f^\prime(u)\xi^2\right]\,\mathrm{d} v_g \geq 0.
	\end{equation}
	
\end{definition}
We note that critical points of the functional $J_\delta$ correspond to weak solutions of the singular problem
\begin{equation}\label{01}\tag{$\mathcal{S}$}
-\textrm{div}(|\nabla_g u|^{p-2}\nabla_g u) = f(u) \quad \text{ in } \quad \mathcal{B}_1\setminus \left\{\mathcal{O}\right\}.
\end{equation}
In particular, if $u$ is a radial local minimizer of $J_\delta$, then $u$ is a semi-stable solution of \eqref{01}. We are going to focus our analysis on the important case $1<p\leq N$, since for $p>N$, it holds $W^{1,p}(\mathcal{B}_1)\hookrightarrow L^\infty(\mathcal{B}_1)$. Precisely, if $N<p\leq +\infty$, from Morrey's inequality we have $\Vert u \Vert_{C^{0,\gamma}(\mathcal{B}_1)}\leq C\Vert u \Vert_{W^{1,p}(\mathcal{B}_1)}$, where $C$ is a constant which depends on $p$ and $N$.

\subsection{Main Results and Comments}

The aim of the paper is twofold. Firstly, to establish a priori estimates for radial semi-stable classical solutions of \eqref{01}. Precisely, we establish $L^\infty$, $L^q$ and $W^{1,q}$ estimates for semi-stable, radially symmetric, and decreasing solutions of \eqref{01} without assuming Dirichlet boundary condition or any other  kind of boundary conditions. 
It should be an interesting question to study similar results for non-radial solutions.
We stress that our results hold for any locally Lipschitz non-linearity $f(s)$ and metric $g$ satisfying \eqref{09}. 

Since the celebrated paper by B. Gidas, W. M. Ni and L. Nirenberg
\cite{GNN} the question of symmetry in non-linear partial differential equations has been
the subject of intensive investigations. In $[16]$, by using a variant of moving planes method, it was established radial symmetry for non-negative solutions of quasilinear elliptic equations defined in geodesic balls of the hyperbolic space $\mathbb{H}^n$ with homogeneous Dirichlet boundary condition. 

\begin{theorem}\label{20}
For any $f(s)$ locally Lipschitz function and $u$ semi-stable solution of \eqref{01}, it holds that:
\begin{enumerate}
\item[$(a)$]\label{37} If $N < p+4p/(p-1)$, then $u\in L^\infty(\mathcal{B}_1)$ and
\[
\Vert u \Vert_{L^\infty(\mathcal{B}_1)}\leq C_{N,p,\alpha,\psi}\Vert u \Vert_{L^{p}(\mathcal{B}_1)}.
\]
\item[$(b)$] If $N \geq p+4p/(p-1)$, then $u\in L^q(\mathcal{B}_1)$ and 
\[
\Vert u \Vert_{L^q(\mathcal{B}_1)} \leq C_{N,\psi,p,q} \Vert u \Vert_{L^p(\mathcal{B}_1)} \quad \mbox{for any} \quad  q<q_0:=\frac{Np}{N-p-2-2\sqrt{\frac{N-1}{p-1}}}.
\]
Moreover, $u\in W^{1,q}(\mathcal{B}_1)$ and $$\Vert u \Vert_{W^{1,q}(\mathcal{B}_1)} \leq C_{N,\psi,p,q} \Vert u \Vert_{L^p(\mathcal{B}_1)} \quad \mbox{for any} \quad  q<q_1:=\frac{Np}{N-2-2\sqrt{\frac{N-1}{p-1}}}.$$ 
\end{enumerate}
\end{theorem}

\begin{remark} For our argument in the proof of Theorem \ref{20} it was crucial the following key estimate for semi-stable solutions of \eqref{01}
\begin{equation}\label{27}
\int_{0}^{\delta}|u_r|^p\psi^{N-1-2\alpha}\mathrm{d}r \leq C_{N,p,\alpha,\psi}\Vert u \Vert_{L^p(\mathcal{B}_1)}^p,
\end{equation}
where $\psi$ is the polar decomposition of $\mathrm{d}s^2$ given in \eqref{09} (see Proposition \ref{23} below). We have proved this estimate by using the radial symmetry of the solution $u$ and by choosing an appropriated test function in the semi-stability inequality \eqref{04}, 
\end{remark}

\begin{remark}
From Theorem \ref{20} $(a)$, one can see that Problem \eqref{01} does not have any singular solution.
\end{remark}

\begin{remark}
	Note that $q_0>p^*=(Np)/(N-p)$ (the critical Sobolev exponent) and $q_1>p$.
Under the hypotheses of Theorem \ref{20} if $N\geq p+4p/(p-1)$ then $u$ belongs to $L^q\left(\mathcal{B}_1\right)$ for all $q<q_0$. Since $q_0$ is greater than the critical Sobolev exponent, from Theorem \ref{20} $(ii)$ we conclude that semi-stable radially symmetric and decreasing weak solutions of \eqref{01} have a better regularity than the one expected by using the classical Sobolev embedding. Moreover, we established better regularity than $W^{1,p}$ for semi-stable solutions to Problem \eqref{01}, since our estimates shows an improvement in the Sobolev space $W^{1,q}$ for $q<q_1$.
\end{remark}

Our second purpose of this work is to apply the elliptic estimates obtained in Theorem \ref{20} to prove regularity results for the following class of quasilinear elliptic problems
\begin{equation}\label{02}
\left\{
\begin{alignedat}{3}
-\textrm{div}(|\nabla_g u|^{p-2}\nabla_g u) = & \, \lambda h(u) & \quad \text{in} & \quad\mathcal{B}_1, \\
u > &\, 0 & \quad \text{in} & \quad \mathcal{B}_1,\\
u = & \, 0  &  \text{on} & \quad \partial\mathcal{B}_1,\\
\end{alignedat}
\right.\tag{$\mathcal{P}_{\lambda}$}
\end{equation}
where $\lambda$ is a positive parameter and $h(s)$ is an increasing $C^1-$function such that  $h(0)>0$ and 
\begin{equation}\label{21}\tag{$H_1$}
\lim_{t\rightarrow +\infty}\frac{h(t)}{t^{p-1}}=+\infty.
\end{equation}

The study of this class of problems with various boundaries conditions has received considerable attention in recent years under the influence of the pioneering works of I.~Gelfand \cite{GEL1959}, D.~Joseph and T.~Lundgren \cite{JOSLUN1972}, J.~Keener and H.~Keller \cite{KEEKEL1974}, M.~Crandall and P.~Rabinowitz \cite{CRARAB1975}, F.~Mignot and J.-P.~Puel \cite{MIGPUE1978}. First, we would like to mention the progress involving Laplacian
\begin{equation}\label{22}
\left\{
\begin{alignedat}{3}
-\Delta u= & \, \lambda h(u) & \quad \text{in} & \quad\Omega, \\
u = & \, 0  &  \text{on} & \quad \partial\Omega,\\
\end{alignedat}
\right.
\end{equation}
where $\Omega$ is a bounded domain of $\mathbb{R}^N$. Non-linear elliptic problems like \eqref{22} appear naturally in several physical phenomena, just to mention some applications, it arises in the theory of non-linear diffusion generated by non-linear sources \cite{JOS1965,JOSCOH1967,JOSSPA1970}, thermal ignition of a chemically active mixture of gases \cite{GEL1959}, membrane buckling \cite{CAL1971} and gravitation equilibrium \cite{CHA1985}. We refer the reader to \cite{DAV2008,DUP2011,espghoguo2009} for a recent survey on this subject.

In recent years, regularity issues about this class of singular elliptic problems have been the focus of an active research area. 
The parameter $\lambda$ measure the non-dimensional strength of the non-linearity. It is well known that if $h$ is super-linear, there exists $\lambda^*\in (0,+\infty)$ such that if $\lambda\in(0,\lambda^*)$, then problem \eqref{22} admits a semi-stable solution $u_\lambda$ and if $\lambda>\lambda^*$, then problem \eqref{22} admits no regular solution. This allows one to define the extremal solution $u^*:=\lim_{\lambda\nearrow\lambda^*}u_\lambda$, which is a weak solution of \eqref{22}. In \cite{NED2000}, G.~Nedev proved regularity results for extremal solutions of \eqref{22} in dimensions $2$ and $3$ and $L^q$ estimates for every $q<N/(N-4)$ when $N\geq 4$ just assuming that $h(s)$ is a positive convex function with $h(0)>0$ and $h^\prime(0)\geq 0$. For a related problem still in the Euclidean case see \cite{CAB2010}, where X.~Cabr\'e assuming $h(s)$ to be a $C^1$ non-decreasing super-linear non-linearity with $h(0)>0$, proved boundedness of the extremal solution for Problem \eqref{22} in dimension $N\leq 4$. 
In dimension $2$ the domain $\Omega$ can be general but, in contrast with Nedev's result, in dimensions $3$ and $4$ the domain is assumed to be convex. After that, X.~Cabr\'{e} and M.~Sanch\'{o}n \cite{CABSAN2013} completed the analysis in \cite{CAB2010} when they proved that if $N\geq 5$ and $\Omega$ is a convex bounded domain of Euclidean space $\mathbb{R}^N$ then the extremal solution belongs to $L^{\frac{2N}{N-4}}$.

Recently, there has been growing interest on singular elliptic partial differential equations on Riemannian manifolds. The problem involving the Laplace-Beltrami operator
\[
\left\{
\begin{alignedat}{3}
-\Delta_g u= & \, \lambda h(u) & \quad \text{in} & \quad\Omega, \\
u = & \, 0  &  \text{on} & \quad \partial\Omega,\\
\end{alignedat}
\right.
\]
where $\Omega$ is a bounded domain was studied recently by D.~Castorina and M.~Sanch\'{o}n in \cite{CASSAN2014} for the inhomogeneous context. They proved qualitative properties for semi-stable solutions and they established $L^\infty$, $L^q$ and $W^{1,q}$ estimates which do not depends on the non-linearity $h(s)$. Furthermore, the authors obtained regularity results for the extremal solution for exponential and power non-linearities. A similar setting has been considered by E.~Berchio, A.~Ferrero and G.~Grillo \cite{BERFERGRI2014} in order to study uniqueness and qualitative properties of radial entire solutions of the Lane--Emden--Fowler equation $-\Delta u=|u|^{m-1}u$ with $m>1$ on certain classes of Cartan--Hadamard manifolds where the so-called Joseph-Lundgren exponent is involved in the stability of solutions. The existence of a stable solution to the semi-linear equation $-\Delta_g u=f(u)$ on a complete, non-compact, boundaryless Riemannian manifold with non-negative Ricci curvature and $f\in C^1$ was studied by A.~Farina, L.~Mari and E.~Valdinoci \cite{FARMARVAL2013}. They classify both the solution and the manifold and also discuss the classification of monotone solutions with respect to the direction of some Killing vector field, in the spirit of a conjecture of De Giorgi. In \cite{MOR2015}, F.~Morabito investigated the existence and uniqueness of positive radial solutions of the problem
\[
\left\{
\begin{alignedat}{3}
\Delta_g u + \lambda u + u^p= & \, 0 & \quad \text{in} &\quad \mathcal{A}, \\
u = & \, 0  &  \text{on} & \quad \mathcal{A},\\
\end{alignedat}
\right.
\]
 when $\lambda <0$, $\mathcal{A}$ is an annular domain in a Riemannian manifold of dimension $N$ endowed with the metric $\mathrm{d}s^2=\mathrm{d}r^2+S^2(r) \mathrm{d} \theta^2$ under suitable assumptions on the function $S^2(r)$. He also show that there exist positive non-radial solutions arising by bifurcation from the radial solution, where $\lambda$ and $p$ are the bifurcation parameters.

Many non-linear problems in physics and mechanics are formulated in equations that
contain the $p-$Laplacian, for example on Non-Newtonian Fluids, Glaceology and Non-linear Elasticity (see \cite{DIAZ1985}). For some problems of non-linear partial differential equations on Riemannian manifold we refer to  \cite{HOLO,KAW,GRO}. Gelfand type problems involving the  $p-$Laplacian in the homogeneous case of the form
\begin{equation}\label{29}
\left\{
\begin{alignedat}{3}
-\Delta_pu = & \, \lambda h(u) & \quad \text{in} & \quad\Omega, \\
u > &\, 0 & \quad \text{in} & \quad \Omega,\\
u = & \, 0  &  \text{on} & \quad \partial\Omega,\\
\end{alignedat}
\right.
\end{equation}
was studied by J.~Garc\'{i}a-Azorero, I.~Peral and J.~Puel \cite{GARPER1992,GARPERPUE1994} where $\Omega$ is a smooth bounded domain of $\mathbb{R}^N$. They proved that for every $p>1$ and $h(s)=e^s$, the extremal solution $u^*$ is an energy solution for every dimension and that it is bounded in some range of dimensions. For a more general non-linearity, X.~Cabr\'{e} and M.~Sanch\'{o}n \cite{CABSAN2007} proved that every semi-stable solution is bounded for a explicit exponent which is optimal for the boundedness of semi-stable solutions and, in particular, it is bigger than the critical Sobolev exponent $p^*-1$. For general $h(s)$ and $p>1$ the interested reader can see \cite{CABCAPSAN2009,CAS2015,SAN2007,WEI2014} for more regularity results about the extremal solution. In \cite{CABCAPSAN2009}, X.~Cabr\'{e}, A.~Capella and M.~Sanch\'{o}n treated the delicate issue about regularity of extremal solutions $u^*$ of \eqref{29} at $\lambda=\lambda^*$ when $\Omega$ is the unit ball of $\mathbb{R}^N$. Among other results, they established pointwise, $L^q$ and $W^{1,q}$ estimates which are optimal and do not depend on the non-linearity $h(s)$.

Furthermore, D.~Castorina and M.~Sanch\'on \cite{CASSAN15} obtain a priori estimates for semi-stable solutions of the reaction-diffusion problem $-\Delta_p u=h(u)$ in $\Omega$ while the reaction term is driven by any positive $C^1$ non-linearity
$h$ and, as a main tool, they develop Morrey-type and Sobolev-type inequalities that involve the functional
\begin{equation}\label{11} I_{p,q}(v;\Omega)=\left(\int_\Omega\left[\left(\frac{1}{p'}|\nabla_{T,v}|\nabla v|^{p/q}|\right)^q+|H_v|^q|\nabla v|^p\right]\,dx\right)^{1/p},\quad\;p,q\geq 1, 
\end{equation} where $v\in C^\infty_0(\overline{\Omega})$. In $\eqref{11}$, $H_v(x)$ denotes the mean curvature at $x$ of the hyper-surface $\{y\in \Omega\: |v(y)|=|v(x)|\}$ and $\nabla_{T,v}$ is the tangential gradient along a level set of $|v|$. In addition to being of independent interest, these geometric inequalities are used, together with judicious choice of test functions in the semi-stability condition, to obtain their a priori estimates for semi-stable solutions. 

In this paper we investigate similar results in the inhomogeneous context of a Riemannian manifold. We use some ideas of \cite{CABSAN2007}, comparison principle for $-\Delta_p$ (because it is uniformly elliptic) and the positivity of the first eigenvalue (as well the corresponding eigenfunction) of $-\Delta_p$ on $\Omega$ (cf. \cite{ANTMUGPUC2007, MAO2014, MON1999}). We point out that the regularity results achieved in this paper represent a geometrical extension of the ones obtained for the Euclidean case in \cite{CABCAPSAN2009}.

Before we state our main result on the regularity of semi-stable solutions for \eqref{02}, let us introduce some basic definitions. We say that $u\in W_{0}^{1,p}(\mathcal{B}_1)$ is a weak solution of \eqref{02} if $h(u)\in L^1(\mathcal{B}_1)$ and
\[
\int_{\mathcal{B}_1}|\nabla u|^{p-2}\nabla u\cdot\nabla\phi\,\mathrm{d}v_g=\int_{\mathcal{B}_1}h(u)\phi\,\mathrm{d}v_g,
\]
for all $\phi\in C_{0}^{\infty}(\mathcal{B}_1)$. Furthermore, by minimal solution we mean smaller than any other super-solution of the Problem \eqref{02} and regular solution means that a weak solution $u$ of \eqref{02} is $C^{1,\beta}(\mathcal{B}_1)$.

Let us state the existence and basic properties of touchdown parameter.

\begin{theorem}\label{26}
	There exist $\lambda^*\in(0,\infty)$ such that
\end{theorem}
\begin{enumerate}
	\item[(i)] For $0< \lambda < \lambda^*$, the problem \eqref{02} has a regular minimal solution $u_{\lambda}$.
	\item[(ii)] For $\lambda>\lambda^*$, \eqref{02} admits no weak solution.
	\item[(iii)] The map $\lambda\rightarrow u_\lambda$ is increasing.
\end{enumerate}

As a consequence of Theorem \ref{26}, the increasing limit 
\[
u^*=\lim_{\lambda\nearrow\lambda^*}u_\lambda
\]
is well defined by the point-wise increasing property. If $u^*$ is a weak solution of $(\mathcal{P}_{\lambda^*})$, then $u^*$ is called the extremal solution. Since the extremal solutions can be obtained as the limit of classical minimal solutions, our next result is useful in order to prove that $u^*$ has the same regularity properties as the ones stated in Theorem \ref{20}. For this, we need to bound $u^{p-1}$ and $h(u)$ in $L^1(\mathcal{B}_1)$ uniformly in $\lambda$. This is possible because we have the growth condition \eqref{21} on $h(s)$ and the radially decreasing property of the minimal solutions $u_\lambda$. Let us now state precisely our results for \eqref{02}.

\begin{theorem}\label{03}
Suppose that $N < p+4p/(p-1)$ and let $h(s)$ be a positive and increasing $C^1-$function satisfying \eqref{21}. Then $u^*$ is a semi-stable solution of $(P_{\lambda^*})$ and $u^*\in L^\infty(\mathcal{B}_1)$.
\end{theorem}

\subsection{Outline} In the next section we bring basic facts about the p-Laplace Beltrami operator which will be used through the paper. In Section \ref{15} we prove existence of extremal parameter $\lambda^*$ and minimal solutions of \eqref{02} for $0<\lambda<\lambda^*$. In Section \ref{24} we prove that key-estimate \eqref{27} by using a suitable choice of test functions under the semi-stability property. In Section \ref{25} we use \eqref{27} to prove our main theorem about regularity for radially symmetric and decreasing semi-stable solutions of Problem \eqref{01} and apply this results for the study of the regularity of extremal solutions of \eqref{02}.

\section{Proof of Theorem \ref{26}}\label{15}
Our first proposition establishes the analog of the classical results for \eqref{02} in the Euclidean case. Using some ideas coming from X.~Cabr\'{e} and M.~Sanch\'{o}n \cite{CABSAN2007} and X.~Luo, D.~Ye and F.~Zhou \cite{LUYEZH2011} we prove the existence of a critical parameter $\lambda^*$ which is related with the resolvability of \eqref{02}.

\begin{proof}[Proof of Theorem \ref{26}]
For $(i)$, let $w \in W_{0}^{1,p}(\mathcal{B}_1)$ a weak solution of $-\mathrm{div}(|\nabla_g w|^{p-2}\nabla_g w) = 1$ in $\mathcal{B}_1$, that is
\[
\displaystyle\int_{\mathcal{B}_1}|\nabla_g w|^{p-2}\nabla_g w\cdot\nabla_g \phi\, \mathrm{d}\sigma = \displaystyle\int_{\mathcal{B}_1}\phi \, \mathrm{d} \sigma,\quad\forall\phi\in C_{0}^{\infty}(\mathcal{B}_1).
\]
We can see that $w \in C^{1,\alpha}(\overline{\mathcal{B}_1})$ by using $C^{1,\alpha}$-regularity results (see \cite{LIE1988, BEN,TOL}). By the Maximum principle \cite[Theorem 3.3]{ANTMUGPUC2007} $w$ is non-negative in $\mathcal{B}_1$. It is easy verify that $0$ is sub-solution of \eqref{02} and if $\lambda \leq \lambda_0:=1/h(\max_{\overline{\mathcal{B}_1}} w),$
\[
\displaystyle\int_{\mathcal{B}_1}|\nabla_g w|^{p-2}\nabla_g  w\cdot\nabla_g\phi\, \mathrm{d}\sigma= \displaystyle\int_{\mathcal{B}_1}\phi\, \mathrm{d}\sigma \geq \displaystyle\int_{\mathcal{B}_1}\lambda h(w)\phi\, \mathrm{d}\sigma
\]
that is, $w$ is a super-solution of \eqref{02}. Thus, for any $\lambda\leq\lambda_0$, Problem \eqref{02} has a weak solution $u\in W_{0}^{1,p}(\mathcal{B}_1)$ given by Sub and Super-solution Method (see \cite{KUR1989}) with $0 \leq u \leq w$ in $\overline{\mathcal{B}_1}$. This implies that $u\in C^{1,\alpha}(\overline{\mathcal{B}_1})$. As any regular solution $u$ of \eqref{02} is also a super-solution for $(P_\mu)$ if $\mu \in (0,\lambda)$, the set of $\lambda$ for which \eqref{02} admits a regular solution is an interval. For $(ii)$ we will show that for $\lambda$ sufficient large, there is no regular solution for \eqref{02}, so $\lambda^* <+\infty.$ It is well known that for the non-linear eigenvalue problem
\[
\left\{
\begin{alignedat}{3}
-\textrm{div}(|\nabla_g v_1|^{p-2}\nabla_g v_1) = & \, \lambda_1 \vert v_1\vert^{p-2}v_1 & \quad \text{in} & \quad \mathcal{B}_1, \\
v_1 = & \, 0  &  \text{in} & \quad \partial\mathcal{B}_1,\\
\end{alignedat}
\right.
\]
there exists a smaller positive and simple eigenvalue $\lambda_1$ with a positive eigenfunction $v_1$ in $\mathcal{B}_1$. Now, suppose that \eqref{02} admits a regular solution $u$ for $\lambda > \lambda_1$. The regularity result in \cite{LIE1988} give that $v_1\in C^{1,\alpha}(\overline{\mathcal{B}_1})$. By homogeneity, we can assume that $\Vert v_1\Vert_{\infty}<h(0)^{\frac{1}{p-1}}$. Note that
$$-\textrm{div}(|\nabla_g v_1|^{p-2}\nabla_g v_1) = \lambda_1 v_{1}^{p-1} \leq \lambda_1 h(0) < \lambda h(u) = -\textrm{div}(|\nabla_g u|^{p-2}\nabla_g u).$$
By the comparison principle \cite{ANTMUGPUC2007} we have that $v_1\leq u$. Let us to take $v_2$ a solution of
\[
\left\{
\begin{alignedat}{3}
-\textrm{div}(|\nabla_g v_2|^{p-2}\nabla_g v_2) = & \, (\lambda_1 + \epsilon) v_{1}^{p-1} & \quad \text{in} & \quad \mathcal{B}_1, \\
v_2 = & \, 0  &  \text{in} & \quad \partial\mathcal{B}_1,\\
\end{alignedat}
\right.
\]
where $\epsilon$ is a positive constant. For $\lambda>\frac{\lambda_1+\epsilon}{h(0)}\max_{\overline{\mathcal{B}_1}}u^{p-1}$ we obtain
$$-\textrm{div}(|\nabla_g v_2|^{p-2}\nabla_g v_2) = (\lambda_1 + \epsilon) v_{1}^{p-1} \leq (\lambda_1 + \epsilon) u^{p-1} \leq \lambda h(u) = -\textrm{div}(|\nabla_g u|^{p-2}\nabla_g u).$$
Using the comparison principle again we obtain $v_1 \leq v_2 \leq u.$ Now, let us define recursively $u_n$ as the unique solution of
\[
\left\{
\begin{alignedat}{3}
-\textrm{div}(|\nabla_g v_n|^{p-2}\nabla_g v_n) = & \, (\lambda_1 + \epsilon) v_{n-1}^{p-1} & \quad \text{in} & \quad \mathcal{B}_1, \\
v_n = & \, 0  &  \text{in} & \quad \partial\mathcal{B}_1.\\
\end{alignedat}
\right.
\]
By comparison principle we obtain $v_1 \leq ... \leq v_{n-1} \leq v_n \leq u \in C^{1,\alpha}(\overline{\mathcal{B}_1}).$ This implies that $v_n\rightharpoonup u_\lambda$ in $W_{0}^{1,p}(\mathcal{B}_1)$ and consequently $u_\lambda$ satisfies
\[
\left\{
\begin{alignedat}{3}
-\textrm{div}(|\nabla_g u_\lambda|^{p-2}\nabla_g u_\lambda) = & \, (\lambda_1 + \epsilon) u_\lambda^{p-1} & \quad \text{in} & \quad \mathcal{B}_1, \\
u_\lambda = & \, 0  &  \text{in} & \quad \partial\mathcal{B}_1,\\
\end{alignedat}
\right.
\]
which is impossible since the first eigenvalue of p-Laplace Beltrami operator is isolated (see \cite{ZHA,LIND}). Define the critical threshold $\lambda^*$ as the supremum of $\lambda>0$ for which \eqref{02} admits a regular solution. Thus we have that $\lambda^*<+\infty.$ Note that, by the construction above, $u_\lambda$ is independent of the choice of the super-solution. Since any regular solution of \eqref{02} is also a super-solution of \eqref{02}, we can conclude that $u_\lambda$ is a regular minimal solution of \eqref{02}. In order to check (iii), let $\lambda\leq\mu$. Thus, $u_\mu$ is a super-solution of \eqref{02}, which implies that $u_\lambda \leq u_\mu$, that is, the map $\lambda\rightarrow u_\lambda$ is increasing.
\end{proof}

\section{A priori estimates}\label{24}

In this section we prove the principal estimate \eqref{27}, which as we already mention it is the crucial in our argument to obtain the regularity of the semi-stable solutions in Theorem \ref{20}. The main idea is to apply an appropriate test function in the stability
inequality (see Proposition \ref{23}). The radial form of \eqref{01} can be written as follows
\begin{equation}\label{05}
-(p-1)|u_r|^{p-2}u_{rr}-\frac{(N-1)\psi^\prime}{\psi}|u_r|^{p-2}u_r = f(u) \quad \text{ with } r\in (0,1).
\end{equation}

Let us discuss a few preliminary estimates which will be used in our argument. In the next Lemma we prove that the second variation of energy associated to \eqref{01} is independent of the non-linearity $f(s)$.

\begin{lemma}\label{08}
Let $u\in W_{r}^{1,p}(\mathcal{B}_1)$ be a semi-stable solution of \eqref{01} satisfying $u_r(r)<0$ for all $r\in (0,1)$. Then, for all radially symmetric function $\eta\in C_{c}^{1}(\mathcal{B}_1\setminus\left\{\mathcal{O}\right\})$ it holds
\[
\displaystyle\int_{\mathcal{B}_1}|u_r|^p\left[ (p-1)|\eta_r|^2+\frac{\partial}{\partial r}\left( \frac{(N-1)\psi^\prime}{\psi} \right)\eta^2\right]\,\mathrm{d} v_g\geq 0.
\]
\end{lemma}

\begin{proof}
We start considering $\eta\in C_{c}^{1}(\mathcal{B}_1\setminus\mathcal{O})$ be a radial function with compact support in $\mathcal{B}_1\setminus\mathcal{O}$ and choosing $\xi=u_r\eta$ as test function in \eqref{04} there holds
\begin{equation}\label{06}
\begin{alignedat}{1}
0\leq \int_{\mathcal{B}_1}&(p-1)|u_r|^{p-2}|\nabla_g(u_r\eta)|^2-f^\prime(u)|u_r|^2\eta^2\,\mathrm{d} v_g\\
&=\int_{\mathcal{B}_1}(p-1)|u_r|^{p}|\nabla_g\eta|^2+(p-1)|u_r|^{p-2}\nabla_g(\eta^2u_r)\nabla_g(u_r) - f^\prime(u)|u_r|^2\eta^2\,\mathrm{d} v_g.
\end{alignedat}
\end{equation}
On the other hand, multiplying \eqref{05} by $(\eta^2u_r\psi^{N-1})_r$, integrating and using integration by parts we are able to compute
\[
\begin{alignedat}{1}
0 = &\int_{0}^{1}(p-1)|u_r|^{p-2}u_{rr}\left( \eta^2 u_r \psi^{N-1} \right)_r+\left[ \frac{(N-1)\psi^\prime}{\psi}|u_r|^{p-2}u_r + g(u) \right]\left( \eta^2 u_r\psi^{N-1} \right)_r\,\mathrm{d} v_g\\
= & \int_{0}^{1}(p-1)|u_r|^{p-2}u_{rr}(\eta^2u_r\psi^{N-1})_r-\left[ \frac{(N-1)\psi^\prime}{\psi}|u_r|^{p-2}u_r + g(u) \right]_r\eta^2u_r\psi^{N-1}\,\mathrm{d} v_g,
\end{alignedat}
\]
which together with $\partial_r(|u_r|^{p-2}u_r)=(p-1)|u_r|^{p-2}u_{rr}$ implies
\[
\begin{alignedat}{2}
0&=\int_{0}^{1}(p-1)|u_r|^{p-2}u_{rr}\partial_r(\eta^2 u_r \psi^{N-1})\,\mathrm{d} r - \int_{0}^{1}\partial_r\left( \frac{(N-1)\psi^\prime}{\psi} \right)|u_r|^{p-2}u_r\eta^2 u_r\psi^{N-1}\,\mathrm{d} r\\
 & -\int_{0}^{1}\frac{(N-1)\psi^\prime}{\psi}(p-1)|u_r|^{p-2}u_{rr}\eta^2 u_r \psi^{N-1}\,\mathrm{d} r - \int_{0}^{1} f^\prime (u)u_r\eta^2 u_r \psi^{N-1}.
\end{alignedat}
\]
Thus
\[
\begin{alignedat}{2}
0&=\int_{0}^{1}(p-1)|u_r|^{p-2}u_{rr}\partial_r\left( \eta^2u_r \right)\psi^{N-1}\,\mathrm{d} r-\int_{0}^{1}\partial_r\left( \frac{(N-1)\psi^\prime}{\psi} \right)|u_r|^{p-2}u_{r}^{2}\eta^2\psi^{N-1}\,\mathrm{d} r\\
 & -\int_{0}^{1}f^\prime(u)\eta^2u_{r}^{2}\psi^{N-1}\,\mathrm{d} r,
\end{alignedat}
\]
which yields
\begin{equation}\label{07}
\int_{\mathcal{B}_1}\partial_r\left(\frac{(N-1)\psi^\prime}{\psi}\right)|u_r|^p\eta^2\,\mathrm{d} v_g=\int_{\mathcal{B}_1}(p-1)|u_r|^{p-2}u_{rr}\partial_r(\eta^2u_r)\,\mathrm{d} v_g - \int_{\mathcal{B}_1}f^\prime(u)\eta^2u_{r}^{2}\,\mathrm{d} v_g.
\end{equation}
Using \eqref{06} and \eqref{07} we have
\[
\begin{alignedat}{2}
0&\leq\int_{\mathcal{B}_1}(p-1)|u_r|^{p-2}u_{r}^{2}|\nabla \eta|^2+ \int_{\mathcal{B}_1}(p-1)|u_r|^{p-2}u_{rr}\nabla(\eta^2u_r)\nabla u_r-f^\prime(u)u_{r}^{2}\eta^2\,\mathrm{d} v_g\\
 & = \int_{\mathcal{B}_1}(p-1)|u_r|^p|\nabla\eta|^2+\int_{\mathcal{B}_1}\frac{\partial}{\partial r}\left(\frac{(N-1)\psi^\prime}{\psi}\right)|u_r|^p\eta^2\,\mathrm{d} v_g\\
 & = \int_{\mathcal{B}_1}|u_r|^p\left[ (p-1)|\eta_{r}|^2+\frac{\partial}{\partial r}\left( \frac{(N-1)\psi^\prime}{\psi}\eta^2 \right) \right]\,\mathrm{d} v_g,
\end{alignedat}
\]
which is the desired conclusion.
\end{proof}

We obtain $L^p-$estimates for the radial derivative of semi-stable solutions of \eqref{01} with the help of Lemma \ref{08}. For that we consider a suitable class of test functions to analyze the inhomogeneous context of a Riemannian manifold assuming that $p\leq N$ and $1\leq\alpha<1+\sqrt{(N-1)/(p-1)}$. To be more precise,

\begin{proposition}\label{23}
Let $u\in W_{r}^{1,p}(\mathcal{B}_1)$ be a semi-stable solution in $\mathcal{B}_1\setminus \mathcal{O}$ of \eqref{01} satisfying $u_r(r)<0$ for $r\in (0,1)$ and $\delta=\delta(\psi)\in (0,1/2)$ such that $\psi^\prime>0$ in $[0,\delta]$.
Then
\[
\int_{0}^{\delta}|u_r|^p\psi^{N-1-2\alpha}\mathrm{d}r \leq C_{N,p,\alpha,\psi}\Vert u \Vert_{L^p(\mathcal{B}_1)}^p
\]
for every $1\leq\alpha<1+\sqrt{(N-1)/(p-1)}$, where $C_{N,p,\alpha,\psi}$ is a constant depending only on $N,$ $p,$ $\alpha$ and $\psi$.
\end{proposition}

\begin{proof}
Using the semi-stability condition of $u$ and applying Lemma \ref{08} with $\psi\eta$ as test function we obtain
\begin{equation}\label{31}
(N-1)\int_{\mathcal{B}_1}\left[ -\psi^{\prime\prime}\psi+\left(\psi^\prime\right)^2 \right]|u_r|^p\eta^2\,\mathrm{d} v_g\leq (p-1)\int_{\Omega}|u_r|^p|(\psi\eta)_r|^2\,\mathrm{d} v_g
\end{equation}
Now, take $\alpha$ satisfying $1\leq\alpha<1+\sqrt{(N-1)/(p-1)}$, $\epsilon\in (0,1)$ sufficiently small and
\[
\eta_\epsilon(r) = \left\{
\begin{alignedat}{3}
\psi^{-\alpha}(\epsilon)-\psi^{-\alpha}(\delta) & \text{\quad for \quad} & 0\leq r \leq \epsilon \\
\psi^{-\alpha}(r)-\psi^{-\alpha}(\delta) & \text{\quad for \quad} & \epsilon<r\leq \delta \\
0 & \text{\quad for \quad} & \delta <r\leq 1
\end{alignedat}
\right.
\]
a Lipschitz function which vanishes on $\partial\mathcal{B}_1$. Choosing $\eta=\eta_\epsilon$ in the inequality \eqref{31} we have
\[
\begin{alignedat}{2}
&(N-1)\left(\int_{0}^{\epsilon}\left[ -\psi^{\prime\prime}\psi+\left(\psi^\prime\right)^2 \right]\eta_{\epsilon}^{2}|u_r|^p\psi^{N-1}\,\mathrm{d} r + \int_{\epsilon}^{\delta}\left[ -\psi^{\prime\prime}\psi+\left(\psi^\prime\right)^2 \right]\eta_{\epsilon}^{2}|u_r|^p\psi^{N-1}\,\mathrm{d} r\right)\\
&\leq (p-1)\left(\int_{\epsilon}^{\delta}\left[(1-\alpha)\psi^{-\alpha}-\psi^{-\alpha}(\delta)\right]^2|u_r|^p\left(\psi^\prime \right)^2\psi^{N-1}\,\mathrm{d}r+\int_{0}^{\epsilon}\eta_{\epsilon}^{2}|u_r|^p\left( \psi^\prime \right)^2\psi^{N-1}\,\mathrm{d} r\right),
\end{alignedat}
\]
which can be written as
\[
\begin{alignedat}{2}
&(N-p)\int_{0}^{\epsilon}\eta_{\epsilon}^{2}|u_r|^p(\psi^\prime)^2\psi^{N-1}\,\mathrm{d}r+(N-1)\int_{\epsilon}^{\delta}(\psi^\prime)^2\eta_{\epsilon}^{2}|u_r|^p\psi^{N-1}\,\mathrm{d}r\leq\\
&(p-1)\int_{\epsilon}^{\delta}\left[(1-\alpha)\psi^{-\alpha}-\psi^{-\alpha}(\delta)\right]^2|u_r|^p\left(\psi^\prime \right)^2\psi^{N-1}\,\mathrm{d}r+(N-1)\int_{0}^{\delta}\psi^{\prime\prime}\psi|u_r|^p\eta_{\epsilon}^{2}\psi^{N-1}\,\mathrm{d}r.
\end{alignedat}
\]
Since $(N-p)\eta_{\epsilon}^{2}|u_r|^p(\psi^\prime)^2\psi^{N-1}\,\mathrm{d}r\geq 0$ we obtain
\[
\begin{alignedat}{2}
(N-1)\int_{\epsilon}^{\delta}|u_r|^p\left( \psi^\prime \right)^2\eta_{\epsilon}^{2}\psi^{N-1}\mathrm{d}r\, &\leq\, (p-1)\int_{\epsilon}^{\delta}|u_r|^p\left( (1-\alpha)\psi^{-\alpha}-\psi^{-\alpha} \right)^2\left( \psi^\prime \right)^2\psi^{N-1}\mathrm{d}r\\
&+(N-1)\int_{0}^{\delta}\psi^{\prime\prime}\psi|u_r|^p\eta_{\epsilon}^{2}\psi^{N-1}\mathrm{d}r
\end{alignedat}
\]
Throughout the proof, $\tilde{C}_{n,p,\alpha}$ (respectively $\tilde{C}_{n,p,\alpha,\psi}$) denote different positive constants depending only on $n$, $p$ and $\alpha$ (respectively on $n$, $p$, $\alpha$, $\psi$). Rewritten the above equation follows
\[
\begin{alignedat}{2}
&\int_{\epsilon}^{\delta}\left( \psi^\prime \right)^2|u_r|^p\psi^{-2\alpha}\psi^{N-1}\mathrm{d}r \leq \tilde{C}_{n,p,\alpha}\left\{\int_{\epsilon}^{\delta}\left( \psi^\prime \right)^2|u_r|^p\psi^{-2\alpha}(\delta)\psi^{N-1}\mathrm{d}r\right.\\
&\left.+\int_{0}^{\delta}\psi^{N-1}|u_r|^p|\psi^{\prime\prime}|\psi\left( \psi^{-2\alpha}+\psi^{-2\alpha}(\delta) \right)\mathrm{d}r+\int_{\epsilon}^{\delta}\left( \psi^\prime \right)^2|u|^p\psi^{-\alpha}\psi^{-\alpha}(\delta)\psi^{N-1}\mathrm{d}r\right\}.
\end{alignedat}
\]
Observe that, by assumption, $\displaystyle\inf_{(0,\delta)}\psi^\prime$ and $\displaystyle\sup_{(0,\delta)}\psi^\prime$ are positive. Now, we can rearrange the terms in the integrals to obtain
\[
\begin{alignedat}{2}
\int_{\epsilon}^{\delta}|u_r|^p\psi^{-2\alpha}\psi^{N-1}\mathrm{d}r &\leq \tilde{C}_{N,p,\alpha}\int_{0}^{\delta}\psi^{N-1}|u_r|^p|\psi^{\prime\prime}|\psi\psi^{-2\alpha}\left( 1+\frac{\psi^{2\alpha}}{\psi^{2\alpha}(\delta)} \right)\,\mathrm{d}r\\
&+\tilde{C}_{N,p,\alpha}\int_{0}^{\delta}\left( \psi^\prime \right)^2|u|^p\psi^{-\alpha}\psi^{-\alpha}(\delta)\psi^{N-1}\left( 1+\frac{\psi^\alpha}{\psi^\alpha(\delta)} \right)\,\mathrm{d}r \\
&\leq\tilde{C}_{N,p,\alpha,\psi}\int_{0}^{\delta}\psi^{N-1}|u_r|^p\psi^{-\alpha}\left\{ 1+\psi^{1-\alpha} \right\}
\end{alignedat}
\]
Taking $\epsilon\rightarrow 0$, follows that
\begin{equation}\label{32}
\int_{0}^{\delta}|u_r|^p\psi^{-2\alpha}\psi^{N-1}\mathrm{d}r \leq \tilde{C}_{N,p,\alpha,\psi}\int_{0}^{\delta}\psi^{N-1}|u_r|^p\psi^{-\alpha}\left\{  1+\psi^{1-\alpha} \right\}\,\mathrm{d}r.
\end{equation}
If we define
\[
\zeta(t)=\frac{\tilde{C}_{N,p,\alpha,\psi}t^{-\alpha}(1+t^{1-\alpha})-\frac{t^{-2\alpha}}{2}}{t^{\frac{N-1}{p-1}}},
\]
using that $1\leq p \leq N$ and $\alpha$ satisfying $1\leq\alpha<1+\sqrt{(N-1)/(p-1)}$, we can check that $\lim_{t\rightarrow +\infty}\zeta(t)=0$ and $\lim_{t\rightarrow 0^+}\zeta(t)=-\infty$. Thus, by a compactness argument, $\zeta(t)$ is bounded from above. This implies that there exists $C_{N,p,\alpha,\psi}>0$ such that
\[
\tilde{C}_{N,p,\alpha,\psi}t^{-\alpha}(1+t^{1-\alpha})\leq\frac{t^{-2\alpha}}{2}+C_{N,p,\alpha,\psi}t^\frac{N-1}{p-1}\quad \forall t>0
\]
and \eqref{32} leads to
\begin{equation}\label{35}
\int_{0}^{\delta}|u_r|^p\psi^{-2\alpha}\psi^{N-1}\mathrm{d}r \leq C_{N,p,\alpha,\psi}\int_{0}^{\delta}|u_r|^p\psi^{(N-1)p/(p-1)}\,\mathrm{d}r.
\end{equation}
On the other hand, since $u$ is radially decreasing follows that
\begin{equation}\label{33}
u^{p}(\delta)\leq C_{N,\psi}\int_{0}^{\delta}u^{p}\psi^{N-1}\,\mathrm{d}r\leq C_{N,\psi}\Vert u \Vert_{L^p(\mathcal{B}_1)}^p
\end{equation}
and using Mean value theorem for some $\tilde{\delta}\in(\delta,2\delta)$ it holds
\begin{equation}\label{34}
-u_r(\tilde{\delta})=\frac{u(\delta)-u(2\delta)}{\delta}\leq \frac{u(\delta)}{\delta}
\end{equation}
Thus, integrating \eqref{05} from $r\in(0,\delta)$ to $\tilde{\delta}$ and using \eqref{34} we obtain
\[
\begin{alignedat}{2}
-\vert u_r(r) \vert^{p-2}u_r(r)\psi^{N-1}(r)&= -\vert u_r(\tilde{\delta}) \vert^{p-2}u_r(\tilde{\delta})\psi^{N-1}(\tilde{\delta})-\int_{s}^{\tilde{\delta}}f(u)\psi^{N-1}\leq\frac{u^{p-1}(\delta)}{\delta^{p-1}}\psi^{N-1}(\tilde{\delta}),
\end{alignedat}
\]
which together \eqref{33} implies
\[
\vert u_r \vert^p\psi^{(N-1)p/(p-1)}\leq \frac{u^p(\delta)}{\delta^p}\psi^{p(N-1)/(p-1)}(\tilde{\delta}) \leq C_{N,\psi}\Vert u \Vert_{L^p(\mathcal{B}_1)}^p.
\]
Integrating from $0$ to $\delta$ to obtain
\[
\int_{0}^{\delta}\vert u_r \vert^p\psi^{(N-1)p/(p-1)}\leq C_{N,\psi}\Vert u \Vert_{L^p(\mathcal{B}_1)}^p
\]
and going back to \eqref{35} we conclude that 
\[
\int_{0}^{\delta}|u_r|^p\psi^{N-1-2\alpha}\mathrm{d}r \leq C_{N,p,\alpha,\psi}\Vert u \Vert_{L^p(\mathcal{B}_1)}^p,
\]
which completes the proof.
\end{proof}

\section{Proof of Main Theorems}\label{25}

\begin{proof}[Proof of Theorem \ref{20}]
Let $\delta\in(0,1/2)$ as in Proposition \ref{23}. Since $u$ is radially symmetric and positive we can check that
\begin{equation}\label{38}
u(\delta)\leq \int_{0}^{\delta}u\psi^{N-1}\mathrm{d}r\leq C_{N,\psi}\Vert u \Vert_{L^1(\mathcal{B}_1)}
\end{equation}
Using H\"older inequality we can estimate
\begin{equation}\label{36}
\begin{alignedat}{2}
\vert u(t) \vert&=\left\vert u(\delta)-\int_{t}^{\delta}u_r\psi^{(N-1-2\alpha)/p}\psi^{(-N+1+2\alpha)/p}\mathrm{d}r \right\vert\\
 &\leq C_{N,\psi}\Vert u \Vert_{L^1(\mathcal{B}_1)}+\left( \int_{0}^{\delta}\vert u_{r}\vert^p\psi^{N-1-2\alpha} \right)^{\frac{1}{p}}\left( \int_{t}^{\delta}\psi^{(-N+1+2\alpha)/(p-1)} \right)^{\frac{p-1}{p}}\\
 &\leq C_{N,\psi}\Vert u \Vert_{L^1(\mathcal{B}_1)}+C_{N,p,\alpha,\psi}\Vert u \Vert_{L^p(\mathcal{B}_1)} \left(\int_{t}^{\delta}\psi^{(-N+1+2\alpha)/(p-1)} \right)^{\frac{p-1}{p}}
\end{alignedat}
\end{equation}
(i) In order to prove $L^\infty$ estimate, observe that, by monotonicity, for all $\delta\leq t<1$ we have
\[
u^p(t)\leq u^p(\delta)\leq C_{N,\psi}\Vert u \Vert_{L^p(\mathcal{B}_1)}^p.
\]
Taking $t=0$ in \eqref{36}, we can analyze the integral and check that
\[
\int_{0}^{\delta}\psi^{(2\alpha-N+1)/(p-1)} < +\infty
\]
when $(2\alpha-N+1)/(p-1)>-1$, that is, $\alpha>(N-p)/2$. Thus, for all $0<t<\delta$ we have
\[
\vert u(t) \vert \leq C_{N,\psi}\Vert u \Vert_{L^1(\mathcal{B}_1)}+C_{N,p,\alpha,\psi}\Vert u \Vert_{L^p(\mathcal{B}_1)},
\]
whenever $\max\left\{ (N-p)/2,1 \right\}<\alpha<1+\sqrt{(N-1)/(p-1)}$. This occurs if, and only if, $N<p+4p/(p-1)$. Therefore, the desired $L^\infty$ estimate \eqref{37} holds true.\\
(ii) On the other hand, since $u$ is decreasing, using \eqref{38} we have
\begin{equation}\label{39}
\left(\int_{\delta}^{1}\vert u \vert^q\psi^{N-1}\mathrm{d}t\right)^{\frac{1}{q}}\leq C_{N,\psi,q} u(\delta)\leq C_{N,\psi,q}\Vert u \Vert_{L^1(\mathcal{B}_1)}.
\end{equation}
Taking $t\in(0,\delta)$ and using \eqref{36} we have
\[
\int_{t}^{\delta}\vert u \vert^q\psi^{N-1}\mathrm{d}t \leq C_{N,\psi,q} \Vert u \Vert_{L^p(\mathcal{B}_1)}^{q}\int_{0}^{\delta}\left[ 1 + \left(\int_{t}^{\delta}\psi^{(-N+1+2\alpha)/(p-1)} \right)^{\frac{p-1}{p}}  \right]^q\psi^{N-1}\mathrm{d}t.
\]
Therefore, if $q<Np/(N-p-2-2\sqrt{(N-1)/(p-1)})$, choosing suitable $\alpha$ such that
\[
\left(\int_{0}^{\delta}\vert u \vert^q\psi^{N-1}\mathrm{d}t\right)^{\frac{1}{q}} \leq C_{N,\psi,q} \Vert u \Vert_{L^p(\mathcal{B}_1)}.
\]
Taking this last inequality combined with \eqref{39} we obtain the desired $L^q$ estimate.\\
Now, we are looking for $W^{1,q}$ estimate. We use similar idea as can be found in the proof of Theorem 1.2 in \cite{CASSAN2014}. For this, observe that every function $u\in W_{r}^{1,p}(\mathcal{B}_1)$ also belongs to the Sobolev space $W^{1,p}(\delta,1)$. Thus, we have
\begin{equation}\label{41}
\int_{\delta}^{1}|u_r|^q\psi^{N-1}\mathrm{d}r\leq C_{N,q,\psi}\int_{\delta}^{1}|u_r|^q\leq C_{N,q,\psi}u^q(\delta)\leq C_{N,q,\psi}\Vert u \Vert_{L^q(\mathcal{B}_1)}^{q}.
\end{equation}
Now, using equation \eqref{05} we have
\[
u_{rr}\leq -\frac{(N-1)\psi^\prime}{\psi} u_r \quad \text{ in } (0,1).
\]
Now, let $\tilde{\delta}\in(\delta,2\delta)$ such that \eqref{34} holds. Integrating the last inequality,
\[
\int_{t}^{\tilde{\delta}}u_{rr}\,\mathrm{d}r \leq -(N-1)\int_{t}^{\tilde{\delta}}\frac{\psi^\prime}{\psi}u_r\,\mathrm{d}r
\]
and using \eqref{38},
\[
\begin{alignedat}{2}
-\frac{u_r(t)}{N-1} &\leq -\frac{u_r(\tilde{\delta})}{N-1}-\int_{t}^{\tilde{\delta}}\frac{|\psi^\prime|}{\psi}u_r\,\mathrm{d}r\\
&\leq C_{N,\psi}\Vert u \Vert_{L^1(\mathcal{B}_1)}	+\int_{t}^{2\delta}\frac{\psi^\prime}{\psi}\psi^{(-N+1+2\alpha)/p}\psi^{(N-1-2\alpha)/p}.
\end{alignedat}
\]
Using H\"older inequality and observing that we can use Proposition \ref{23} because we can take our $\delta$ sufficiently small, follows
\[
\begin{alignedat}{2}
-u_r(t) &\leq C_{N,p,\psi}\Vert u \Vert_{L^p(\mathcal{B}_1)}\left( 1+ \left(\int_{t}^{2\delta}\left(\frac{\psi^\prime}{\psi}\right)^{p^\prime}\psi^{p^\prime(-N+1+2\alpha)/p}\right)^{\frac{1}{p^\prime}}\right)
\end{alignedat}
\]
for all $\alpha\in [1,1+\sqrt{(N-1)/(p-1)}]$. Thus, for $s\in (0,\delta)$,
\[
\begin{alignedat}{2}
\int_{s}^{\delta}|u_r|^q\psi^{N-1}\,\mathrm{d}r & \leq C_{N,p,\psi}\Vert u \Vert_{L^p(\mathcal{B}_1)}^{q}\left( 1+ \left(\int_{t}^{2\delta}\left(\frac{\psi^\prime}{\psi}\right)^{p^\prime}\psi^{p^\prime(-N+1+2\alpha)/p}\right)^{\frac{1}{p^\prime}}\right)^q
\end{alignedat}
\]
Now, observe that
\[
\left( \left( \int_{t}^{2\delta} \left(\psi^\prime\right)^\frac{1}{p^\prime}\psi^{\frac{1}{p-1}(-N+1+2\alpha-p)}\right)^{\frac{1}{p^\prime}}\right)^q<C_{N,p,q,\psi}<+\infty
\]
since $q<Np/(N-2-2\sqrt{(N-1)/(p-1)})$. Thus
\begin{equation}\label{40}
\int_{s}^{\delta}|u_r|^q\psi^{N-1}\,\mathrm{d}r \leq C_{N,p,q,\psi}\Vert u \Vert_{L^p(\mathcal{B}_1)}^{q}.
\end{equation}
Using equations \eqref{41} and \eqref{40} we finish the proof.
\end{proof}

\begin{proof}[Proof of Theorem \ref{03}]
Let $\lambda\in (0,\lambda^*)$. There exists $\rho_\lambda\in (1/2,1)$ such that mean value property holds, that is,
\[
\frac{\partial u_\lambda}{\partial r}(\rho_\lambda)=\frac{u_\lambda(1/2)-u_\lambda(1)}{1/2}.
\]
Since $u_\lambda$ is decreasing, (see proof of Theorem \ref{26}), we have
\[
\left[ \frac{\partial u_\lambda}{\partial r}(\rho_\lambda) \right]^{p-1}=\left[ 2u_\lambda(1/2) \right]^{p-1}\leq C_{N,p,\psi}\Vert u_{\lambda}^{p-1} \Vert_{L^1(B_{1/2})}.
\]
Thus
\begin{equation}\label{16}
\Vert \psi^{N-1}|\frac{\partial u_\lambda}{\partial r}|^{p-1} \Vert_{L^\infty(B_{1/2})}\leq C_{N,p,\psi}\Vert u_{\lambda}^{p-1} \Vert_{L^1(B_{1/2})}.
\end{equation}
follows by monotonicity. By using $\phi(r)=\min\left\{1,(2-4r)^+\right\}$ as test function and \eqref{16} we obtain
\begin{equation}\label{19}
\Vert \lambda h(u_\lambda) \Vert_{L^1(B_{1/4})}\leq C_{N,p,\psi}\int_{1/4}^{1/2}\psi^{N-1}\vert\frac{\partial u_\lambda}{\partial r}\vert^{p-1}\,\mathrm{d} r\leq C_{N,p,\psi}\Vert u_{\lambda}^{p-1} \Vert_{L^1(B_{1/2})}.
\end{equation}
Using the assumption \eqref{21}, given $\delta>0$ we have for any $\lambda\in (\lambda^*/2,\lambda^*)$ and for all $t>0$, 
\[
\lambda h(t)\geq \frac{1}{\delta}t^{p-1}-C_\delta,
\]
where $C_\delta$ does not depends on $\lambda$. With this
\begin{equation}\label{17}
\Vert u_{\lambda}^{p-1} \Vert_{L^1(B_{1/4})}\leq C_{N,p,\psi}\delta\Vert u_{\lambda}^{p-1} \Vert_{L^1(B_{1/2})}+C_\delta.
\end{equation}
Since $u_\lambda$ is decreasing follows that
\begin{equation}\label{18}
\Vert u_{\lambda}^{p-1} \Vert_{L^1(B_{1/2}\setminus \overline{B}_{1/4})}\leq C_{N,p,\psi}u_{\lambda}^{p-1}(1/4)\leq C_{N,p,\psi}\Vert u_{\lambda}^{p-1} \Vert_{B_{1/4}}.
\end{equation}
Now, take $\delta$ sufficiently small and combine \eqref{17} with \eqref{18} to obtain
\[
\Vert u_{\lambda}^{p-1} \Vert_{L^1(B_{1/4})}\leq C,
\]
where $C$ is a constant independent of $\lambda$. Repeating the argument in \eqref{18} we are able to obtain an estimate uniform in $\lambda$ for $\Vert u_{\lambda}^{p-1} \Vert_{L^1(B_1)}$. Using this in \eqref{19} we obtain a estimate for $\Vert h(u_\lambda) \Vert_{L^1(B_{1/4})}$. Again by monotonicity we can apply the same argument used above to control $\Vert h(u_\lambda) \Vert_{L^1(B_1)}$ uniformly in $\lambda$. Thus
\begin{equation}\label{10}
\Vert u_{\lambda}^{p-1} \Vert_{L^1(B_1)}+\Vert h(u_\lambda) \Vert_{L^1(B_1)}\leq C,
\end{equation}
where $C$ is a constant independent of $\lambda$. Observe that every radial function $u\in W^{1,p}(\mathcal{B}_1)$ also belongs to the Sobolev space $W^{1,p}(\delta,1)$ in one dimension for a given $\delta \in (0,1).$ Using the Sobolev embedding in one dimension, $u$ becomes a continuous function of $r=\mathrm{dist}(x,\mathcal{O})\in [\delta,1]$ and
\[
|u(1)|\leq C_{N,p}\Vert u \Vert_{W^{1,p}(\mathcal{B}_1)}
\]
In view of this estimate, we can assume that $u>0=u(1)\text{ in }\mathcal{B}_1$. Take $\alpha$ satisfying $1\leq\alpha<1+\sqrt{(N-1)/(p-1)}$ and using Proposition \ref{23},
\[
\begin{alignedat}{3}
\int_{\mathcal{B}_1}&|u_r|^p\psi^{-2\alpha}\,\mathrm{d} v_g \leq C_{N,p,\psi}\int_{\mathcal{B}_1} |u_r|^p \,\mathrm{d} x = C_{N,p,\psi}\int_{\mathcal{B}_{r_0}} |u_r|^p \,\mathrm{d} x + C_{N,p,\psi}\int_{\mathcal{B}_1\setminus\overline{\mathcal{B}_{r_0}}} |u_r|^p \,\mathrm{d} x.&
\end{alignedat}
\]
Now, choose $r_0$ such that $2C_{N,p,\psi}\leq \psi^{-2\alpha}$ in $r\in (0,r_0)$ to obtain
\[
C_{N,p,\psi}\int_{\mathcal{B}_{r_0}} |u_r|^p \,\mathrm{d} x\leq\frac{1}{2}\int_{\mathcal{B}_1} \psi^{-2\alpha}|u_r|^p \,\mathrm{d} x,
\]
which implies
\[
C_\psi\int_{\mathcal{B}_1}|u_r|^p\,\mathrm{d} x\leq \int_{\mathcal{B}_1}\psi^{-2\alpha}|u_r|^p\,\mathrm{d} x \leq C_{N,p,\psi}\int_{\mathcal{B}_1\setminus\overline{\mathcal{B}_{r_0}}}|u_r|^p\,\mathrm{d} x.
\]
Since $u$ is decreasing we have that
\[
u(r_0)^{p-1}\leq C_{N,p}\Vert u^{p-1} \Vert_{L^1(\mathcal{B}_{r_0})}.
\]
Thus,
\[
\begin{alignedat}{2}
\int_{\mathcal{B}_1\setminus\overline{\mathcal{B}_{r_0}}}|u_r|^p\,\mathrm{d} x & = C_{N,\psi}\int_{r_0}^{1}|u_r|^p\psi^{N-1}\,\mathrm{d} r\\
& \leq C_{N,\psi}\Vert \psi^{N-1}|u_r|^{p-1} \Vert_{L^\infty(\mathcal{B}_1)}\int_{r_0}^{1}-u_r\,\mathrm{d} r\\
& \leq C_{N,p,\psi}\Vert h(u) \Vert_{L^1(\mathcal{B}_1)}\Vert u^{p-1} \Vert_{L^1(\mathcal{B}_1)}^{\frac{1}{p-1}}.
\end{alignedat}
\]
We can conclude that
\[
\begin{alignedat}{2}
\int_{\mathcal{B}_1}|u_r|^p\,\mathrm{d} x & \leq C_{N,p,\psi}\int_{\mathcal{B}_1\setminus\overline{\mathcal{B}_{r_0}}}|u_r|^p\,\mathrm{d} x\\
 & \leq C_{N,p,\psi}\Vert h(u) \Vert_{L^1(\mathcal{B}_1)}\Vert u^{p-1} \Vert_{L^1(\mathcal{B}_1)}^{\frac{1}{p-1}}.
\end{alignedat}
\]
By \eqref{10} we deduce a bound for $\Vert u_\lambda \Vert_{W^{1,p}(B_1)}.$ By using the compactness and since $u_\lambda\rightarrow u^*$ as $\lambda\rightarrow\lambda^*$ follows that $u^*\in W_{0}^{1,p}(B_1)$. We can pass to the limit and conclude that $u^*$ is a weak solution of \eqref{02}. It is clear that $u^*$ is radially symmetric and decreasing. By Fatou's Lemma we obtain that $u^*$ is semi-stable. Finally, we can pass to the limit and the regularity statement follows as a consequence of Theorem \ref{20}.
\end{proof}

\begin{corollary}
	The extremal solution $u^*$ has the same regularity stated in Theorem \ref{20}.
\end{corollary}

\begin{proof}
The proof follows straightforward by using above estimates and passing to the limit as $\lambda\rightarrow \lambda^*$.
\end{proof}


\begin{thebibliography}{100}

\bibitem{ANTMUGPUC2007}
\textrm{P.~Antonini, D.~Mugnai, P.~Pucci},
\textit{Quasilinear elliptic inequalities on complete Riemannian manifolds.}
J. Math. Pures Appl. \textbf{87} (2007) 582--600.

\bibitem{BERFERGRI2014}
\textrm{E.~Berchio, A.~Ferrero, G.~Grillo,}
\textit{Stability and qualitative properties of radial solutions of the Lane-Emden-Fowler equation on Riemannian models,}
J. Math. Pure. Appl. \textbf{102} (2014) 1--35.

\bibitem{CABSAN2007}
\textrm{X.~Cabr\'{e}, M.~Sanch\'{o}n,}
\textit{Semi-stable and extremal solutions of reaction equations involving the p-Laplacian,}
Commun. Pure Appl. Anal. \textbf{6} (2007) 43--67.

\bibitem{CABCAPSAN2009}
\textrm{X.~Cabr\'{e}, A.~Capella, M.~Sanch\'{o}n},
\textit{Regularity of radial minimizers of reaction equations involving the p-Laplacian,}
Calc. Var. Partial Differential Equations \textbf{34} (2009) 475--494.

\bibitem{CAB2010}
\textrm{X. Cabr\'{e}},
\textit{Regularity of minimizers of semilinear elliptic problems up to dimension $4$,}
Comm. Pure Appl. Math. \textbf{63} (2010) 1362--1380.

\bibitem{CABSAN2013}
\textrm{X. Cabr\'{e}, M. Sanch\'{o}n},
\textit{Geometric-type Sobolev inequalities and applications to the regularity of minimizers,}
J. Funct. Anal. \textbf{264} (2013) 303--325.

\bibitem{CAL1971}
\textrm{A.~Callegari, E.~Reiss, H.~Keller},
\textit{Membrane buckling,}
Comm. Pure Appl. Math. \textbf{24} (1971) 499--527.

\bibitem{CASSAN2014}
\textrm{D.~Castorina, M.~Sanch\'{o}n,}
\textit{Regularity of stable solutions to semilinear elliptic equations on Riemannian models,}
Adv. Nonlinear Anal. \textbf{4} (2015) 295--309.

\bibitem{CASSAN15}
\textrm{D.~Castorina, M.~Sanch\'{o}n,}
\textit{Regularity of stable solutions of p-Laplace equations through geometric Sobolev type inequalities,}
J. Eur. Math. Soc. (JEMS) \textbf{17} (2015), 2949--2975. 

\bibitem{CAS2015}
\textrm{D. Castorina},
\textit{Regularity of the extremal solution for singular p-Laplace equations,}
Manuscripta Math. \textbf{146} (2015) 519--529.

\bibitem{CHA1985}
\textrm{S.~Chandrasekhar,}
\textit{An introduction to the study of stellar structure,}
Dover Publ. Inc. New York, 1985.

\bibitem{CRARAB1975}
\textrm{M.~Crandall, P.~Rabinowitz},
\textit{Some continuation and variational methods for positive solutions of nonlinear elliptic eigenvalue problems,}
Arch. Rational Mech. Anal. \textbf{58} (1975) 207--218.
	
\bibitem{DAV2008}
\textrm{J.~D\'avila},
\textit{Singular solutions of semi-linear elliptic problems,}
Handbook of differential equations: stationary partial differential equations. Vol. VI, Handb. Differ. Equ., Elsevier/North-Holland, Amsterdam, 2008, pp. 83--176.

\bibitem{DIAZ1985}
\textrm{J.~I.~D\'{\i}az}, 
\textit{Nonlinear partial differential equations and free boundaries. Vol. I. Elliptic equations}. Research Notes in Mathematics, 106. Pitman (Advanced Publishing Program), Boston, MA, 1985.


\bibitem{BEN}
\textrm{E.~DiBenedetto,}
\textit{$C^{1+\alpha}$ local regularity of weak solutions of degenerate elliptic equations,}
Nonlinear Anal. \textbf{7} (1983), 827--850. 

\bibitem{doO_Costa_2016}
\textrm{J.~M.~do~\'{O}, R.~da~Costa}, 
\textit{Symmetry properties for nonnegative solutions of non-uniformly elliptic equations in the hyperbolic space}. J. Math. Anal. Appl. \textbf{435} (2016), 1753--1771.


\bibitem{DUP2011}
\textrm{L.~Dupaigne},
\textit{Stable Solutions of Elliptic Partial Differential Equations,}
Chapman \& Hall/CRC, Boca Raton, 2011

\bibitem{espghoguo2009}
\textrm{P.~Esposito, N.~Ghoussoub, N.~Guo},
\textit{Mathematical analysis of partial differential equations modeling electrostatic MEMS,}
Courant Lecture Notes in Mathematics, 20. Courant Institute of Mathematical Sciences, New York; American Mathematical Society, Providence, RI, 2010.

\bibitem{FARMARVAL2013}
\textrm{A.~Farina, L.~Mari, E.~Valdinoci},
\textit{Splitting theorems, symmetry results and overdetermined problems for Riemannian manifolds,}
Comm. Partial Differential Equations \textbf{38} (2013) 1818--1862 .

\bibitem{GARPER1992}
\textrm{J. Garc\'{i}a-Azorero, I. Peral},
\textit{On an Emden- Fowler type equation,}
Nonlinear Anal. \textbf{18} (1992) 1085--1097.

\bibitem{GARPERPUE1994}
\textrm{J. Garc\'{i}a-Azorero, I. Peral, J. P. Puel},
\textit{Quasilinear problems with exponential growth in the reaction term,}
Nonlinear Anal. \textbf{22} (1994) 481--498.

\bibitem{GEL1959}
\textrm{I.~M.~Gelfand},
\textit{Some problems in the theory of quasilinear equations,}
Section 15, due G.~I.~Barenblatt, American Math. Soc. Transl. \textbf{29} (1963) 295--381; Russian original: Uspekhi Mat. Nauk. \textbf{14} (1959) 87--158.

\bibitem{GRO}
\textrm{J. Grosjean}:
\textit{p-Laplace operator and diameter of manifolds,}
Ann. Global Anal. Geom. \textbf{28} (2005), 257--270.

\bibitem{GNN}
\textrm{B. Gidas, W. M. Ni, L. Nirenberg}, 
\textit{Symmetry and related properties via the maximum principle}, 
Comm. Math. Phy. 68 (1979), 209--243.

\bibitem{HOLO}
\textrm{I. Holopainen}:
\textit{Asymptotic Dirichlet problem for the p-Laplacian on Cartan-Hadamard manifolds,}
Proc. Amer. Math. Soc. \textbf{130} (2002), 3393--3400. 

\bibitem{JOS1965}
\textrm{D.~Joseph}:
\textit{Non-linear heat generation and the stability of the temperature distribution in conducting solids.}
Int. J. Heat Mass Transfer \textbf{8} (1965) 281--288.

\bibitem{JOSCOH1967}
\textrm{D.~Joseph, D.~Cohen}:
\textit{Some positone problems suggested by nonlinear heat generation.}
J. Math. Mech. \textbf{16} (1967) 1361--1376.

\bibitem{JOSSPA1970}
\textrm{D.~Joseph, E.~Sparrow}:
\textit{Nonlinear diffusion induced by nonlinear sources.}
Quart. Appl. Math. \textbf{28} (1970) 327--342.

\bibitem{JOSLUN1972}
\textrm{D.~Joseph, T.~Lundgren}:
\textit{Quasilinear Dirichlet problem driven by positive sources.}
Arch. Rational Mech. Anal. \textbf{49} (1972) 241--269.

\bibitem{KAW}
\textrm{S. Kawai, N. Nakauchi}:
\textit{The first eigenvalue of the p-Laplacian on a compact Riemannian manifold.}
Nonlinear Anal. \textbf{55} (2003), 33--46.

\bibitem{KEEKEL1974}
\textrm{J.~Keener, H.~Keller},
\textit{Positive solutions of convex nonlinear eigenvalue problems.}
J. Differential Equations \textbf{16} (1974) 103--125.

\bibitem{KUR1989}
\textrm{T.~Kura,}:
\textit{The weak Supersolution-Subsolution Method for second order quasilinear elliptic equations.}
Hiroshima Math. J. \textbf{19} (1989), 1--36.

\bibitem{LIE1988}
\textrm{G.~M.~Lieberman},
\textit{Boundary regularity for solutions of degenerate elliptic equations,}
Nonlinear Anal. \textbf{12} (1988) 1203--1219.

\bibitem{LIND}
\textrm{P. Lindqvist},
\textit{On the equation $div(|\nabla u|^{p-2}\nabla u) +\lambda|u|^{p-2}u = 0$.}
Proc. Amer. Math. Soc. \textbf{109} (1990), 157--164

\bibitem{LUYEZH2011}
\textrm{X. Luo, D. Ye, F. Zhou},
\textit{Regularity of the extremal solution for some elliptic problems with singular nonlinearity and advection.}
J. Differential Equations \textbf{251} (2011) 2082--2099.

\bibitem{MAO2014}
\textrm{J.~Mao},
\textit{Eigenvalue inequality for the p-Laplacian on a Riemannian manifold and estimates for the heat kernel,}
J. Math. Pures Appl. \textbf{9} (2014) 372--393.

\bibitem{MIGPUE1978}
\textrm{F.~Mignot, J.-P.~Puel},
\textit{Sur une classe de probl\`{e}mes non lin\'{e}aires avec non lin\'{e}arit\'{e} positive, croissante, convexe,}
Publications de I'UER de Math. de Lille \textbf{1} (1978).

\bibitem{MON1999}
\textrm{M.~Montenegro},
\textit{Strong maximum principles for super-solutions of quasilinear elliptic equations,}
Nonlinear Anal. \textbf{37} (1999) 431--448.

\bibitem{MOR2015}
\textrm{F.~Morabito,}
\textit{Radial and non-radial solutions to an elliptic problem on annular domains in Riemannian manifolds with radial symmetry,}
J. Differential Equations \textbf{258} (2015) 1461--1493.

\bibitem{NED2000}
\textrm{G.~Nedev},
\textit{Regularity of the extremal solution of semilinear elliptic equations.}
C. R. Acad. Sci. Paris S\'{e}r.~I Math. \textbf{300} (2000) 997--1002. 

\bibitem{SAN2007}
\textrm{M. Sanch\'{o}n},
\textit{Boundedness of the extremal solution for some p-Laplacian problems,}
Nonlinear Anal. \textbf{67} (2007) 281--294.

\bibitem{TOL}
\textrm{P.~Tolksdorf},
\textit{Regularity for a more general class of quasilinear elliptic equations,}
J. Differential Equations \textbf{51} (1984), 126--150. 

\bibitem{WEI2014}
\textrm{L. Wei},
\textit{Boundedness of the extremal solution for some p-Laplacian problems,}
Math. Slovaca \textbf{64} (2014) 379--390.

\bibitem{ZHA}
\textrm{H. Zhang},
\textit{Lower bounds for the first eigenvalue of the p-Laplace operator on compact manifolds with nonnegative Ricci curvature,}
Adv. Geom. 7 (2007), no. 1, 145--155.

\end{thebibliography}
\end{document}